\documentclass[12pt]{amsart}

\usepackage{amssymb,latexsym}

\usepackage{enumerate}

 \usepackage[french,english]{babel}
 
 \usepackage{xcolor}

\usepackage[inner=1.0in,outer=1.0in,bottom=1.0in, top=1.0in]{geometry}

\allowdisplaybreaks

\makeatletter

\@namedef{subjclassname@2010}{  \textup{2010} Mathematics Subject Classification}

\makeatother
\newtheorem{thm}{Theorem}[section]

\newtheorem{cor}[thm]{Corollary}
\newtheorem{lem}[thm]{Lemma}
\newtheorem{pro}[thm]{Proposition}
\theoremstyle{definition}

\newtheorem*{rem}{Remark}

\numberwithin{equation}{section}

\newcommand{\X}{\mathbb{X}}

\newcommand{\re}{\textup{Re}}
\newcommand{\im}{\textup{Im}}

\newcommand{\B}{\mathcal B}

\newcommand{\Lv}{\mathbf{L}}

\newcommand{\ub}{\mathbf{u}}
\newcommand{\vb}{\mathbf{v}}
\newcommand{\xb}{\mathbf{x}}
\newcommand{\yb}{\mathbf{y}}
\newcommand{\zb}{\mathbf{z}}
\newcommand{\kb}{\mathbf{k}}
\newcommand{\lb}{\mathbf{l}}

\newcommand{\me}{\textup{meas}}

\newcommand\be{\begin{equation}}
\newcommand\ee{\end{equation}}

\def\u{\mathbf u}
\def\v{\mathbf v}

\def\E{\mathbb E}

 \def\rt{ {\mathcal{R}_T}}
 \def\Kcal{  \mathcal{K}}
 \def\Hcal{  \mathcal{H}}

 \def\Phrand{ \Phi_T^{\mathrm{rand}}}

 \def\Phrhat{\widehat{\Phi}_T^{\mathrm{rand}}}

\newcommand{\newabstract}[1]{%
  \par\bigskip
  \csname otherlanguage*\endcsname{#1}%
  \csname captions#1\endcsname
  \item[\hskip\labelsep\scshape\abstractname.]
}

\begin{document}

\baselineskip=17pt

\title[Central Limit Theorem of $L$-functions]{Selberg's central limit theorem of $L$-functions near the critical line}

\author[Yoonbok Lee]{Yoonbok Lee }
\address{Department of Mathematics \\ Research Institute of Basic Sciences \\ Incheon National University \\ 119 Academy-ro, Yeonsu-gu, Incheon, 22012 \\ Korea}
\email{leeyb@inu.ac.kr, leeyb131@gmail.com}

\date{\today}

\begin{abstract} 
We find an asymptotic expansion of a multi-dimensional version of Selberg's central limit theorem for $L$-functions on $ \sigma= \frac12 + ( \log T)^{-\theta}$ and $ t \in [ T, 2T]$, where $ 0 < \theta < \frac12 $ is a constant. 
\end{abstract}

\keywords{Central limit theorem, joint distribution of $L$-functions}

\subjclass[2010]{11M41.}

\maketitle

\section{Introduction}\label{Introduction}

 Selberg's central limit theorem says that the function 
$$ \frac{ \log \zeta( \sigma+it) }{ \sqrt{ \pi \sum_{p < t } p^{- 2 \sigma}}} $$
has a  Gaussian distribution in the complex plane for $ \frac12 \leq \sigma \leq \sigma_T (\theta) $, where
$$   \sigma_T := \sigma_T ( \theta) :=   \frac12 + \frac{1}{ ( \log T)^\theta} $$
for $ \theta > 0 $ throughout the paper. See \cite[Theorem 6.1]{Ts} for a proof and \cite{RS} for a simple proof for the real part. It also holds for other $L$-functions. See \cite[Theorem 2]{Sel} for a general statement.

When $\sigma = \sigma_T $ and $ T \leq t \leq 2T$, we have more precise estimations for the distribution of $ \log \zeta( \sigma + i t)$ in \cite{HL} and \cite{Le5} as follows.
\begin{thm}\cite[Theorem 1.2 and Lemma 2.3]{Le5} \label{thm 1}
Let $ 0 < \theta < \frac12$, $ a<b$ and $ c<d$ be real numbers. There exist constants $ \epsilon, \kappa > 0$ and a sequence $ \{ d_{k,\ell} \}_{k, \ell \geq 0 }$ of real numbers such that
\begin{multline}\label{eqn zeta asymp}
\frac1T \me\{ t \in [T, 2T] : \frac{ \log \zeta ( \sigma_T + it )}{ \sqrt{ \pi \psi_T}} \in [a,b] \times [c,d] \} \\
= \sum_{ k + \ell \leq \epsilon \psi_T } \frac{ d_{k , \ell} }{ \sqrt{ \psi_T}^{k + \ell}}  \int_a^b e^{- \pi  u^2 } \Hcal_k ( \sqrt{ \pi }u ) du \int_c^d  e^{- \pi  v^2 }   \Hcal_\ell ( \sqrt{ \pi } v )  dv + O \bigg(  \frac{1}{ ( \log T)^\kappa}\bigg) 
\end{multline}
as $ T \to \infty$, where $\mathrm{meas}$ denotes the Lebesgue measure on $\mathbb{R}$, 
$$ \psi_T := \sum_p \sum_{k \geq 1}  \frac{1}{ k^2 p^{2 k \sigma_T}} $$
 and $ \Hcal_n ( x) $ is the $n$-th Hermite polynomial defined by 
\begin{equation}\label{def Hcal}
\Hcal_n (x) := (-1)^n e^{x^2 } \frac{ d^n}{dx^n} ( e^{-x^2 }) .
\end{equation}
Moreover, $d_{0,0} = 1 $, $ d_{k,\ell}= 0 $ for $ k+\ell = 1,2$ and $ d_{k, \ell} = O( \delta_0^{-k-\ell})$ for some $ \delta_0 > 0 $ and all $ k, \ell$.
\end{thm}
The leading term of the expansion in \eqref{eqn zeta asymp} is 
$$ \int_a^b e^{- \pi  u^2 }   du \int_c^d  e^{- \pi  v^2 }     dv ,$$
which is   Gaussian, and the lower order terms may be evaluated using
$$ \int_a^b e^{- \pi u^2 } \Hcal_k ( \sqrt{ \pi}u) du =    \frac{-1}{ \sqrt{ \pi}} \bigg( e^{ - \pi b^2 } \Hcal_{k-1} ( \sqrt{ \pi} b ) - e^{- \pi a^2 } \Hcal_{k-1} ( \sqrt{ \pi} a ) \bigg) $$
for $ k \geq 1  $. Note that the sequence $\{ d_{k,  \ell}  \} $ is defined by the generating series (2.19) in \cite{Le5} and 
$ \psi_T = \theta \log \log T + O(1) $ by the prime number theorem.
It might be interesting to compare the asymptotic expansion in \eqref{eqn zeta asymp} with an Edgeworth expansion in the probability theory. See \cite[Chapter 7]{C} for more information. 

In this paper, we    generalize Theorem \ref{thm 1} to a multi-variate setting for the $L$-functions $L_1$, \ldots, $L_J$ satisfying the following assumptions:
\begin{enumerate}
\item[A1:] (Euler product)  For $ j = 1, \ldots , J $ and $\re(s)>1$ we have 
$$ L_j  ( s) = \prod_p \prod_{i=1}^d \bigg(   1 -   \frac{ \alpha_{j,i}(p)}{p^s} \bigg)^{-1} , $$
where $ | \alpha_{j,i} (p)  | \leq p^{\eta}$ for some fixed $ 0 \leq \eta < \frac12 $ and for every $ i = 1, \ldots , d.$

 \item[A2:] (Analytic continuation) For $ j= 1, \ldots, J$, each $L_j$ has an analytic continuation to $\mathbb{C}$ except for a finite number of poles on $\re(s) = 1$. 

\item[A3:] (Functional equation) The functions $L_1, L_2, \dots, L_J$ satisfy the same functional equation
  $$ \Lambda_j(s) = \omega \overline{ \Lambda_j( 1- \bar{s})} ,$$  
where
$$ \Lambda_j(s) := L_j(s) Q^s \prod_{\ell=1}^k \Gamma ( \lambda_{\ell} s+\mu_{\ell} ) , $$  
$ | \omega| =1 $, $Q>0$, $ \lambda_{\ell}>0 $ and $\mu_{\ell} \in \mathbb{C} $ with $ \re ( \mu_{\ell} ) \geq 0 $.
\item[A4:] (Ramanujan hypothesis on average)
$$ \sum_{ p \leq x } \sum_{i=1}^d | \alpha_{j,i }(p) |^2 = O( x^{1+\epsilon})$$
holds  for every $ \epsilon>0$ and for every $ j = 1, \ldots , J $ as $ x \to \infty $.

\item[A5:] (Zero density hypothesis) Let $N_{f } ( \sigma, T )$ be the number of zeros of $f (s)$ in $\re(s) \geq \sigma$ and $ 0 \leq \im(s) \leq T$. Then there exist positive constants $\kappa_1, \kappa_2$ such that for every  $j= 1, \ldots, J$ and all $\sigma\geq  \frac12 $ we have 
$$   N_{L_j } ( \sigma, T ) \ll T^{1 - \kappa_1(\sigma -  \frac12 ) } (\log T)^{\kappa_2}.$$

\item[A6:] (Selberg orthogonality conjecture) By assumption A1 we can write 
 $$ \log L_j(s) = \sum_p \sum_{k=1}^\infty   \frac{ \beta_{L_j} (p^k)}{ p^{ks}}  . $$
Then for all $1\leq j, k \leq J$, there exist constants $ \xi_j >0$ and $ c_{j,k}$ such that
$$  \sum_{p \leq x} \frac{ \beta_{L_j}(p)  \overline{\beta_{L_k}(p) } }{p} = \delta_{j,k} \xi_j \log \log x + c_{j,k} + O \bigg( \frac{1}{ \log x} \bigg),$$ 
  where $ \delta_{j,k} = 0 $ if $ j \neq k $ and $ \delta_{j,k} = 1 $ if $ j = k$. 

\end{enumerate}
The assumptions A1--A6 are standard and expected to hold for all $L$-functions arising from automorphic representation for $GL(n)$. In particular, they are verified by $GL(1)$ and $GL(2)$ $L$-functions, which are the Riemann zeta function, Dirichlet $L$-functions, $L$-functions attached to Hecke holomorphic or Maass cusp forms. Assumption A5 is weaker than the Riemann hypothesis, but it is strong enough to find a short Dirichlet approximation to each $\log L_j (\sigma_T + it )$ for almost all $ t \in [T, 2T]$. For example, see \cite[Lemma 4.2]{LL} for a proof. Assumption A6 insures the statistical independence of the $\log L_j ( \sigma_T + it ) $ for $ j = 1, \ldots , J$. 

Assuming assumptions A1--A6 for $L_1, \ldots, L_J$, we want to find an asymptotic expansion for
\begin{equation}\label{eqn joint dist log Lj}
 \frac1T \me\{ t \in [T, 2T] : \frac{ \log L_j  ( \sigma_T + it )}{ \sqrt{ \pi \psi_{j,T}}} \in [a_j ,b_j] \times [c_j,d_j]  \mathrm{~for~all~} j = 1, \ldots , J \},
 \end{equation}
 where
 \begin{equation}\label{def psijT}
  \psi_{j,T} := \xi_j \theta \log \log T  
  \end{equation}
with the constants $\xi_j$ in assumption A6 and $a_j, b_j, c_j, d_j$ are real numbers for all $j = 1 , \ldots, J$. Let
$$ \Lv(s) : =\Big(\log |L_1(s)|, \dots, \log |L_J(s)|, \arg L_1(s), \dots, \arg L_J(s) \Big) $$
and 
$$  \rt := \prod_{j=1}^J [ a_j \sqrt{ \pi \psi_{j, T}} , b_j  \sqrt{ \pi \psi_{j, T}}]  \times \prod_{j=1}^J [ c_j \sqrt{ \pi \psi_{j, T}} , d_j  \sqrt{ \pi \psi_{j, T}}] , $$
then  \eqref{eqn joint dist log Lj} equals to  
$$ \Phi_T (\rt ):= \frac1T \mathrm{meas}  \{ t\in[T, 2T] :  \Lv (\sigma_T +it )\in \rt  \} .$$ 
\begin{thm}\label{thm 2}
Let $ 0< \theta < \frac12$. Assume assumptions A1--A6 for $L_1, \ldots, L_J $. Then there exist  constants $\epsilon, \kappa > 0$ and a sequence  $ \{ b_{\kb, \lb}\}$ of real numbers such that
\begin{multline}\label{eqn thm 2}
   \Phi_T  ( \rt )     =     \sum_{   \Kcal( \kb+\lb) \leq \epsilon \log\log T  }       b_{\kb, \lb}              \prod_{j=1}^J     \frac{1}{  \sqrt{\psi_{j,T} }^{k_j + \ell_j  }   }   \\
   \times  \prod_{j=1}^J \bigg(  \int_{a_j}^{b_j}  e^{ - \pi u^2 } \Hcal_{k_j } ( \sqrt{\pi}u )du      \int_{c_j}^{d_j}  e^{ - \pi v^2 } \Hcal_{\ell_j } ( \sqrt{\pi}v )dv     \bigg)    
     + O \bigg(   \frac{1}{ ( \log T)^\kappa} \bigg) ,
 \end{multline}
where $ \kb = ( k_1 , \ldots, k_J) $ and $ \lb = ( \ell_1 , \ldots, \ell_J)$ are vectors in $(\mathbb{Z}_{\geq 0})^J $ and $\Kcal(\kb  ) := k_1  + \cdots +k_J   $. 
Moreover, $b_{0,0}= 1 $, $ b_{\kb, \lb} = 0 $ if $\Kcal(\kb + \lb)  = 1$ and $b_{\kb + \lb } = O(  \delta_0^{-\Kcal(\kb+\lb)})$ for some $ \delta_0 > 0$ and all $ \kb, \lb$.
\end{thm}
Theorem \ref{thm 2} will be proved in the beginning of Section \ref{sec 2}. Theorem \ref{thm 2} is essentially the same as Theorem 2.1 in \cite{He1}, but it looks that the expansion in Theorem \ref{thm 2} is longer. Moreover, since the paper \cite{He1} contains only a sketched proof, our proof should be useful.

Unlike $d_{k,\ell} $ in Theorem \ref{thm 1}, $b_{\kb, \lb } $ in Theorem \ref{thm 2} may not be zero for $ \Kcal(\kb + \lb ) = 2 $. One reason is that $\psi_T$ in Theorem \ref{thm 1} and $ \psi_{j, T}$ in Theorem \ref{thm 2} are different up to a constant order, even though they are asymptotically same. Moreover, when $ J > 1 $, there are additional terms  essentially   from the constants $c_{j,k}$ in assumption A6. 

Since the leading term in \eqref{eqn thm 2} is  Gaussian and the other nonvanishing terms are 
$ O \big( \frac{1}{ \log \log T} \big)$, we obtain the following corollary.
\begin{cor}
Let $0 < \theta < \frac12 $. Assume assumptions A1--A6 for $L_1, \ldots, L_J$. Then we have
$$    \Phi_T  ( \rt )     =          \prod_{j=1}^J  \bigg(    \int_{a_j}^{b_j}  e^{ - \pi u^2 }  du      \int_{c_j}^{d_j}  e^{ - \pi v^2 }  dv     \bigg)         + O \bigg(   \frac{1}{ \log \log T  } \bigg) . $$
\end{cor}

   \begin{rem}
  The current method is limited to $ 0< \theta < \frac12$, since our proof of Theorem \ref{thm 2} depends on \eqref{eqn discrepancy lemma}. In a forthcoming paper \cite{Le6} we extend Theorem \ref{thm 2} for $\sigma_T $ closer to $ \frac12 $ by proving \eqref{eqn discrepancy lemma} for 
$$  \frac12 + \frac{( \log \log T)^2}{ \log T } \leq    \sigma_T   \leq \frac12 + \frac{1}{\log \log T}.  $$
 A new ingredient therein is the second moment estimation of $ \log L(s)$ instead of approximations to each $ \log L_j (s)$ by Dirichlet polynomials. Further extension toward $ \frac12$ should be more interesting but challenging, since we need to estimate a contribution from the nontrivial zeros of $L(s)$. 
\end{rem}

We will prove theorems and propositions in Section \ref{sec 2} and lemmas in Section \ref{sec 3}. We conclude the introduction with a summary of notations:
\begin{itemize}
\item $ \sigma_T = \sigma_T (\theta) = \frac12 + \frac{1}{ ( \log T)^\theta}$.
\item $\kb=(k_1 , \ldots, k_J) $ and $\lb = ( \ell_1, \ldots , \ell_J) $ are vectors in $( \mathbb{Z}_{\geq 0 })^J$.
\item $ \ub= ( u_1 , \ldots , u_J) $, $\vb=(v_1 , \ldots , v_J) $, $\xb = (x_1 , \ldots , x_J) $ and $\yb= ( y_1 , \ldots, y_J ) $ are vectors in $\mathbb{R}^J$.
\item $ \zb = ( z_1 , \ldots , z_J) = \xb + i \yb $ and $ \bar{\zb} =  ( \overline{z_1}  , \ldots ,  \overline{ z_J} )=\xb - i \yb$ are   vectors in $\mathbb{C}^J$.
\item $\kb! := k_1 ! \cdots k_J! $ and  $\Kcal(\kb) := k_1 + \cdots + k_J $.
\item $ \xb^\kb := x_1^{k_1} \cdots x_J^{k_J} $.
\item $\xb \cdot \ub  = \sum_{j=1}^J x_j u_j $, $ || \zb || = \sqrt{ \sum_{j=1}^J |z_j|^2  } = \sqrt{ \sum_{j=1}^J ( x_j^2 + y_j^2 )} $.
\end{itemize}

\section{Estimates on random model}\label{sec 2}

 We define the random vector
$$ \Lv(\sigma, \X)=\bigg(\log |L_1(\sigma, \X)|, \dots, \log |L_J(\sigma, \X)|, \arg L_1(\sigma, \X), \dots, \arg L_J(\sigma, \X) \bigg)$$
for $ \sigma > \frac12$, where each $L_j ( \sigma, \X)$ is defined by the product
\begin{equation}\label{def Lj random}
 L_j ( \sigma, \X) = \prod_p \prod_{i=1}^d  \bigg( 1 - \frac{ \alpha_{j,i}(p) \X(p) }{p^\sigma} \bigg)^{-1}
 \end{equation}
and $\{ \X (p) \}_p $ is a sequence of independent random variables, indexed by the prime numbers, and uniformly distributed on the unit circle $\{ z \in \mathbb{C} : |z|=1 \}$. The product converges almost surely for $ \sigma > \frac12$ by Kolmogorov's three series theorem. 

Define a probability measure
\be\label{Phi rand def}
\Phrand  (\B ):=
 \mathbb{P}(\Lv(\sigma_T , \X) \in \B) 
 \ee
 for  a Borel set $ \B$ in $\mathbb R^{2J}$. By \cite[Theorem 2.3]{LL}
we have
\begin{equation}\label{eqn discrepancy lemma}
\Phi_T (\rt  ) =  \Phrand (\rt ) +  O(  ( \log T)^{ (\theta - 1)/2}  \log\log T ) 
\end{equation}
 for $0 <  \theta <  \frac12 $. It means that the distribution of  $\Lv ( \sigma_T + it)$ is well approximated by the distribution of its random model $\Lv ( \sigma_T , \X )$ when $0 <  \theta <  \frac12 $. 
 Thus, Theorem \ref{thm 2} is an immediate consequence of the following theorem and \eqref{eqn discrepancy lemma}.
 \begin{thm}\label{thm 3}
Let $  \theta  > 0 $.  Assume assumptions A1--A6 for $L_1, \ldots, L_J $. Then there exist constants $\epsilon, \kappa > 0$ and a sequence  $ \{ b_{\kb, \lb}\}$ of real numbers such that
\begin{multline*} 
   \Phrand  ( \rt )     =     \sum_{ \Kcal( \kb+ \lb ) \leq \epsilon \log \log T   }       b_{\kb, \lb}              \prod_{j=1}^J      \frac{1}{  \sqrt{\psi_{j,T} }^{k_j + \ell_j  }   } \\
   \times \prod_{j=1}^J \bigg(  \int_{a_j}^{b_j}  e^{ - \pi u^2 } \Hcal_{k_j } ( \sqrt{\pi}u )du      \int_{c_j}^{d_j}  e^{ - \pi v^2 } \Hcal_{\ell_j } ( \sqrt{\pi}v )dv     \bigg)   
     + O \bigg(   \frac{1}{ ( \log T)^\kappa} \bigg) .
 \end{multline*}
Moreover, $b_{0,0}= 1 $, $ b_{\kb, \lb} = 0 $ if $\Kcal(\kb + \lb) = 1$ and $b_{\kb + \lb } = O(  \delta_0^{-\Kcal(\kb+\lb)})$ for some $ \delta_0 > 0$ and all $ \kb, \lb$.
\end{thm}

 In \cite[Section 7]{LL} we find that the measure $ \Phrand $ is absolutely continuous and it has a density function $H_T  (\ub, \vb)$ such that
 \begin{equation}\label{eqn Phrand HT}
  \Phrand ( \rt ) = \iint_\rt  H_T ( \ub, \vb ) d\ub d\vb . 
  \end{equation}
  Hence, Theorem \ref{thm 3} follows from \eqref{eqn Phrand HT} and the following proposition, which upgrades \cite[Lemma 7.4]{LL}.
\begin{pro}\label{prop}
Let $  \theta > 0 $.  Assume assumptions A1--A6 for $L_1, \ldots, L_J $.  There exist constants $ \epsilon, \kappa >0$ and a sequence  $ \{ b_{\kb, \lb}\}$ of real numbers such that 
\begin{align*}
  H_T ( \ub, \vb) = &    \sum_{\Kcal(\kb+ \lb) \leq \epsilon \log \log T }    b_{\kb, \lb }     \prod_{j =1}^J   \frac{1}{ \pi    \sqrt{ \psi_{j, T}}^{k_j + \ell_j +2}    }    e^{-  \frac{  u_j ^2  + v_j^2  }{\psi_{j,T} }}   \Hcal_{k_j}  \bigg( \frac{  u_j }{ \sqrt{ \psi_{j,T } }} \bigg)  \Hcal_{\ell_j}  \bigg( \frac{  v_j }{ \sqrt{ \psi_{j,T } }} \bigg) \\ 
  & + O \bigg(  \frac{1}{ ( \log T)^\kappa} \bigg)  .
  \end{align*}  
  Moreover, $b_{0,0}= 1 $, $ b_{\kb, \lb} = 0 $ if $\Kcal(\kb + \lb)  = 1$ and $b_{\kb + \lb } = O(  \delta_0^{-\Kcal(\kb+\lb)})$ for some $ \delta_0 > 0$ and all $ \kb, \lb$.
 \end{pro}
 
To prove Proposition \ref{prop}, it requires to understand the Fourier transform
$$ \Phrhat (\xb , \yb ) := \int_{ \mathbb{R}^{2J} }  e^{ 2 \pi i  (\xb \cdot \ub + \yb \cdot \vb) } d\Phrand ( \u, \v) $$
for $\xb , \yb \in \mathbb{R}^J$. By the definition of $\Phrand$ in \eqref{Phi rand def},  we have  
 \begin{equation}\notag
\begin{split}
\Phrhat ( \xb, \yb) = \E \Bigg[   \exp \Bigg( 2 \pi  i & \sum_{j=1}^J \big(  x_j \log |   L_j   ( \sigma_T, \X )|   +   y_j \arg L_j ( \sigma_T, \X )   \big) \Bigg) \Bigg].
\end{split}\end{equation}
By assumptions A1 and A6 we see that 
\begin{equation}\label{eqn Ljpk}
 \beta_{L_j} (p^k ) =  \frac1k \sum_{i=1}^d \alpha_{j,i}(p)^k . 
 \end{equation}
By \eqref{eqn Ljpk} and \eqref{def Lj random} we have
$$ \log L_j ( \sigma , \X )   = \sum_p    \sum_{k=1 }^\infty  \frac{  \beta_{L_j} ( p^k ) \X(p)^k }{ p^{k \sigma  }}  . $$
Define
 \begin{equation}\label{def gjps}
  g_{j,p} (  \sigma   ) :=    \sum_{k=1 }^\infty  \frac{  \beta_{L_j} ( p^k ) \X(p)^k }{ p^{k \sigma  }}   , 
  \end{equation}
then we have
 \begin{equation} \label{eqn Phrhat prod 1}
\Phrhat ( \xb, \yb)   = \prod_p \varphi_{p, \sigma_T} (\xb, \yb) ,
 \end{equation}
where 
$$
\varphi_{p, \sigma} (\xb, \yb):= \mathbb{E} \left[ \exp\left(2 \pi i  \sum_{j=1}^J \big( x_j \re\left( g_{j,p} (  \sigma  ) \right)   +     y_j \im \left(   g_{j,p} (  \sigma  ) \right) \big) \right)   \right] $$
 for each prime $p$. Let $ \zb = ( z_1 , \ldots , z_J ) = \xb + i \yb $, then we find that
 $$ \varphi_{p, \sigma } ( \xb , \yb ) = \mathbb{E} \left[ \prod_{j=1}^J  e^{  \pi i    \overline{z_j }  g_{j,p} (  \sigma  ) } e^{ \pi i   z_j \overline{  g_{j,p} (  \sigma  ) }   }    \right]  .$$
By expanding the $2J$ exponential functions into power series we obtain
\begin{align*}
\varphi_{p, \sigma} (\xb, \yb) 
&=  \sum_{ \kb , \lb \in (\mathbb{Z}_{ \geq 0 })^J }  \frac{  (\pi i )^{\Kcal ( \kb + \lb ) } \overline{\zb}^{ \kb }  \zb^{  \lb } }{\kb! \lb!} \E \bigg[  \prod_{j=1}^J  g_{j,p} (  \sigma  )^{k_j }  \overline{g_{j,p} (  \sigma  )}^{\ell_j } \bigg] 
\end{align*}
with notations for vectors in the end of Section \ref{Introduction}. 
  It is easy to see that the expectation
   \begin{equation}\label{def apskl}
    A_{p, \sigma } (\kb,\lb) :=  \E \bigg[  \prod_{j=1}^J g_{j,p} (  \sigma )^{k_j }  \overline{ g_{j,p} (  \sigma  ) }^{\ell_j } \bigg]
    \end{equation}
satisfies $ A_{p, \sigma}(0,0)= 1$ and $ A_{p, \sigma}(0, \kb)= A_{p,\sigma}(\kb, 0)= 0 $ for $ \kb \neq 0 $. Thus, we obtain 
\be\label{PSIA}
 \varphi_{p, \sigma } (\xb, \yb)= 1+ R_{p, \sigma}(\zb),
 \ee
where 
\begin{equation}\label{def Rpstzb}
  R_{p, \sigma }(\zb) := \sum_{\kb \neq 0}  \sum_{\lb \neq 0 }    \frac{  (\pi i )^{\mathcal{K}( \kb + \lb ) } \overline{\zb}^{ \kb }  \zb^{  \lb } }{\kb! \lb!} A_{p,\sigma } (\kb,\lb).
  \end{equation}
Hence, by \eqref{eqn Phrhat prod 1} and \eqref{PSIA} we have
\begin{equation}\label{eqn Phrhat prod 2}
       \Phrhat( \xb, \yb) = \prod_p  ( 1+ R_{p, \sigma_T }(\zb ) ). 
       \end{equation}
To compute the product in \eqref{eqn Phrhat prod 2}, it requires the following lemma.
\begin{lem}\label{lemma Rpst}
There exists a constant $\delta_1 >0$ such that
  $$ |R_{p, \sigma_T} ( \zb) | \leq \frac12 $$
   for every prime $p$ and $ || \zb || \leq \delta_1 $. 
\end{lem}  
See Section \ref{sec proof lemma Rpst} for a proof. By Lemma \ref{lemma Rpst} we have
\begin{equation}\label{eqn Phrhat prod 3} \begin{split}
 \Phrhat(\xb, \yb)  & =  \exp \bigg(  \sum_p   \log  ( 1+ R_{p, \sigma_T} ( \zb ))  \bigg)  \\
 & =  \exp \bigg(  \sum_p  \sum_{m=1}^\infty \frac{ (-1)^{m-1}}{ m}   R_{p, \sigma_T} ( \zb )^m \bigg) 
 \end{split}
 \end{equation} 
 for $|| \zb || \leq \delta_1  $.  By \eqref{def Rpstzb} the sum $ \sum_p  \sum_{m=1}^\infty \frac{ (-1)^{m-1}}{ m}   R_{p, \sigma} ( \zb )^m $ has a power series representation in $ z_1 , \ldots , z_J,  \overline{z_1} , \ldots, \overline{z_J} $, so let $B_{\sigma} ( \kb , \lb)$ be the coefficients such that
\begin{equation}\label{def Bskl}
 \sum_{ \kb \neq 0 } \sum_{ \lb \neq 0 }  B_{\sigma} (\kb , \lb ) \overline{\zb}^\kb \zb^\lb =     \sum_p   \sum_{ m=1}^{\infty }    \frac{ (-1)^{m-1}}{m}      R_{p, \sigma}(\zb)^m  . 
 \end{equation}
Define  $ I_{n, \sigma} ( \zb)$ for each $ n \geq 2 $ by the sum of the degree $n$ terms in the above sum, i.e.,
\begin{equation}\label{def Inszb}
   I_{n, \sigma}(\zb) := \sum_{ \substack{ \kb  , \lb \neq 0  \\    \Kcal(\kb + \lb ) = n  }    }    B_{\sigma} (\kb , \lb ) \overline{\zb}^\kb \zb^\lb  .
   \end{equation}
We see that $ I_{n, \sigma} ( \zb)$ is a homogeneous polynomial in $x_1 , \ldots , x_J, y_1 , \ldots , y_J$ of degree $n$, and that
\begin{equation}\label{eqn Phrhat InsT}
 \Phrhat ( \xb, \yb)=\exp \left(      \sum_{ n=2}^\infty        I_{n, \sigma_T } ( \zb)  \right) 
 \end{equation}
for $ ||\zb || \leq \delta_1 $ by \eqref{eqn Phrhat prod 3}--\eqref{def Inszb}. We find an asymptotic formula for $I_{n, \sigma_T}( \zb)$ as $ T \to \infty$ in the following lemma.
\begin{lem}\label{lemma asymp InsT}
There are complex numbers $ C_{j_1 , j_2 }$ such that
\begin{equation}\label{eqn I2sTzb asymp}
  I_{2, \sigma_T } (\zb) = - \pi^2   \sum_{j =1}^J \psi_{j, T} |z_j|^2  +  \sum_{j_1 , j_2  =1}^J     C_{j_1 , j_2 } \overline{z_{j_1}}z_{j_2}  + O  \bigg(  \frac{ \log \log T }{ ( \log T)^\theta } \bigg)     
  \end{equation}
for $ || \zb || \leq \delta_1 $, where $ \psi_{j,T}$ is defined in \eqref{def psijT} and $ \overline{C_{j_1, j_2 }} = C_{j_2 , j_1 }$. For $ n \geq 3 $, there is a constant $ C=C_{J,d,\eta}  >0$ such that  
$$ |  I_{n, \sigma}(\zb)| \leq C^n || \zb||^n $$
for $ \sigma \geq \frac12$ and
$$ |  I_{n, \sigma_T }(\zb)-I_{n,1/2}(\zb)| \leq  \frac{ C^n || \zb||^n}{( \log T)^\theta} . $$
\end{lem}

See Section \ref{sec proof lemma asymp InsT} for a proof. Define
\begin{equation}\label{def QTzb}
 Q_T (\zb ) := - \pi^2  \sum_{j =1}^J  \psi_{j, T}  |z_j|^2,  
 \end{equation}
\begin{equation}\label{def I2zb}
 I_2 ( \zb ) :=   \sum_{j_1 , j_2  =1}^J     C_{j_1 , j_2 } \overline{z_{j_1}}z_{j_2}  
 \end{equation}
and 
\begin{equation}\label{def Inzb}
 I_n ( \zb ) := I_{n, 1/2 } ( \zb) 
 \end{equation}
for $ n > 2 $. By \eqref{def I2zb} and the Cauchy-Schwarz inequality we obtain 
$$ | I_2 ( \zb) | \leq J ( \max_{j_1, j_2} |C_{j_1, j_2}|) || \zb ||^2  .$$
By this inequality, \eqref{def Inzb} and Lemma \ref{lemma asymp InsT} we have
\begin{equation}\label{Inzb bound 1}
 | I_n ( \zb ) | \leq 2^{-n}
 \end{equation}
for $ n \geq 2 $ and $ || \zb || \leq \delta_2 $, where 
\begin{equation}\label{def delta2}
 \delta_2 := \min\bigg\{ \delta_1 , \frac{1}{2C} , \frac{1}{2 \sqrt{ J  \max_{j_1, j_2} |C_{j_1, j_2}|}  }  \bigg\} .  
 \end{equation}
It follows from \eqref{eqn Phrhat InsT}, Lemma \ref{lemma asymp InsT} and \eqref{def QTzb}--\eqref{Inzb bound 1} that 
\begin{equation}\label{eqn Phrhat sum 1}\begin{split}
\Phrhat ( \xb, \yb) & =\exp \left[   Q_T ( \zb)  +     \sum_{ n=2}^\infty        I_{n } ( \zb) + O\bigg(  \frac{\log \log T }{ ( \log T)^\theta} \bigg)   \right]  \\
& = e^{Q_T (\zb)}  \bigg(  \sum_{r=0}^\infty \frac{1}{ r!}  \bigg(    \sum_{ n=2}^\infty        I_{n  } ( \zb) \bigg)^r  + O\bigg(  \frac{\log \log T }{ ( \log T)^\theta} \bigg) \bigg) 
\end{split}\end{equation}
for $ || \zb || \leq \delta_2 $.
Note that each $I_n ( \zb)$ is a homogeneous polynomial in $ x_1 , \ldots , x_J, y_1, \ldots, y_J$ of degree $ n$ and does not depend on $ T$.
Since the sum  $  \sum_{r=0}^\infty \frac{1}{ r!}  \big(    \sum_{ n=2}^\infty        I_{n  } ( \zb) \big)^r  $ is a power series in $ \xb $ and $\yb$, we let $\{ b_{\kb, \lb} \} $ be a sequence of complex numbers such that
\begin{equation}\label{def bkl}
 G(\xb, \yb):=  \sum_{\kb, \lb } ( 2 \pi i )^{\Kcal( \kb+ \lb)} b_{\kb, \lb} \xb^\kb \yb^\lb  =   \sum_{r=0}^\infty \frac{1}{ r!}  \bigg(    \sum_{ n=2}^\infty        I_{n  } ( \zb) \bigg)^r  .
  \end{equation}
 Then the $b_{\kb, \lb}$ satisfy the following properties.
 \begin{lem} \label{lem bkl}
Let  $ \delta_3$ be a constant satisfying $ 0 < \delta_3 < \frac{ \pi}{\sqrt{J}} \delta_2$, then $b_{\kb, \lb}$ is a real number and 
\begin{equation}\label{bkl bound}
  | b_{\kb, \lb} | \leq \frac{\sqrt{e}}{\delta_3^{\Kcal( \kb + \lb ) }} 
  \end{equation}
  for every $\kb, \lb$.     In particular,  
$ b_{0,0}=1 $ and $ b_{\kb, \lb} = 0 $ if $\Kcal( \kb + \lb ) = 1 $. 
\end{lem}

See Section \ref{sec proof lem bkl} for a proof. 
  The infinite sum over $\kb, \lb$ in \eqref{def bkl} can be approximated by its partial sum. We shall prove a quantitative version. Let $ \epsilon>0$.
By \eqref{def bkl} and \eqref{Inzb bound 1}  we have
\begin{align*}
   \bigg|  \sum_{ \Kcal(\kb+ \lb) > \epsilon \log \log T } ( 2 \pi i )^{\Kcal( \kb+ \lb)} b_{\kb, \lb} \xb^\kb \yb^\lb  \bigg|   & \leq  \sum_{r=1}^\infty \frac{1}{ r!} \sum_{ \substack{  n_1 , \ldots , n_r \geq 2 \\   n_1 + \cdots + n_r > \epsilon \log \log T}   }    \bigg( \frac12 \bigg)^{ n_1 + \cdots + n_r }  \\
   & \leq  \sum_{r=1}^\infty \frac{1}{ r!} \sum_{m>  \epsilon \log \log T } \frac{1}{2^m}   \sum_{ \substack{  n_1 , \ldots , n_r \geq 2 \\   n_1 + \cdots + n_r =m}   }     1  
   \end{align*}
   for $|| \zb || \leq \delta_2 $.
We substitute $ n_j$ by  $ n'_j + 2  $ for $ j = 1, \ldots , r$ in the last sum, then the last sum equals to the number of nonnegative integers $n'_1, \ldots, n'_r$ such that $ n'_1 + \ldots + n'_r = m-2r$, which equals to $ { m-r-1 } \choose {r-1 }$. Thus, the above sum is 
 \begin{align*}
  & \leq  \sum_{r=1}^\infty \frac{1}{ r!} \sum_{m>  \epsilon \log \log T } \frac{1}{2^m} { { m-r-1 } \choose {r-1 }}  \leq \sum_{r=1}^\infty \frac{1}{ r!} \sum_{m>  \epsilon \log \log T } \frac{1}{2^m} \frac{ m^{r-1}}{(r-1)! }\\
  & \leq \sum_{m>  \epsilon \log \log T } \frac{1}{2^m} \sum_{n=0 }^\infty   \frac{ m^{n}}{(n!)^2 }    \leq \sum_{m>  \epsilon \log \log T } \frac{1}{2^m} \bigg(  \sum_{n=0 }^\infty   \frac{ \sqrt{m}^{n}}{n! } \bigg)^2 = \sum_{m>  \epsilon \log \log T } \frac{e^{2 \sqrt{m}}}{2^m}   \\
  & \leq \sum_{m > \epsilon \log \log T} \bigg( \frac23 \bigg)^m \leq 3 \bigg( \frac23 \bigg)^{\epsilon \log \log T} \ll \frac{ 1}{ ( \log T)^\kappa}
   \end{align*}
 with a constant $\kappa  \leq \epsilon \log \frac32$. It follows from these estimates, \eqref{eqn Phrhat sum 1}, \eqref{def bkl} and Lemma \ref{lem bkl} we obtain the following proposition.
\begin{pro}\label{prop 2}
Let $\delta_2 $ be the constant defined in \eqref{def delta2}. Let $ \kappa $ and $\epsilon$ be constants such that $0 < \kappa < \theta $ and $\kappa  \leq \epsilon \log \frac32$. Let $\{ b_{\kb, \lb } \} $ be a sequence of real numbers defined by its generating series \eqref{def bkl}. Then  
$$ \Phrhat ( \xb, \yb ) = e^{Q_T (\zb)}  \bigg(   \sum_{\Kcal(\kb+ \lb)\leq \epsilon \log \log T } ( 2 \pi i )^{\Kcal( \kb+ \lb)} b_{\kb, \lb} \xb^\kb \yb^\lb      + O\bigg(  \frac{1 }{ ( \log T)^\kappa } \bigg) \bigg) $$
holds for $ || \zb || \leq \delta_2 $.
\end{pro}

 We are ready to prove Proposition \ref{prop}. 
The density function $H_T ( \ub, \vb) $ of the measure $\Phrand$ is the inverse Fourier transform of $ \Phrhat$, so that
$$ H_T ( \ub, \vb) = \int_{\mathbb{R}^{J}}\int_{\mathbb{R}^{J}}  \Phrhat ( \xb, \yb ) e^{ - 2 \pi i ( \xb \cdot \ub + \yb \cdot \vb ) } d \xb d \yb . $$
Let $\delta_4$ be a constant such that $ 0 < \delta_4 \leq \min\{ \delta_2 ,  \frac{ \delta_3}{ 4 \pi }\} $. 
By Lemma 7.1 and (7.14) in \cite{LL} we find that
$$ H_T ( \ub, \vb) = \iint_{ || \zb || \leq \delta_4 } \Phrhat ( \xb, \yb ) e^{ - 2 \pi i ( \xb \cdot \ub + \yb \cdot \vb ) } d \xb d \yb  + O \bigg(  \frac{1}{ ( \log T)^\kappa} \bigg)$$
for some $\kappa > 0 $. See the proof of \cite[Lemma 7.4]{LL} for a detail.

By Proposition \ref{prop 2} we have
$$ H_T ( \ub, \vb) =   \sum_{\Kcal(\kb+ \lb) \leq \epsilon \log \log T }  ( 2 \pi i )^{\Kcal(\kb + \lb)}  b_{\kb, \lb }  \iint_{ || \zb || \leq \delta_4 } e^{Q_T (\zb)  - 2 \pi i ( \xb \cdot \ub + \yb \cdot \vb ) }    \xb^\kb  \yb^\lb      d \xb d \yb  + O \bigg(  \frac{1}{ ( \log T)^\kappa} \bigg) $$
for some $ \epsilon, \kappa >0$.
Let $\xi_{min} = \min_{j \leq J} \xi_j > 0 $, then we have
\begin{align*}
\bigg|  \iint_{ || \zb || \geq \delta_4 } & e^{Q_T (\zb)  - 2 \pi i ( \xb \cdot \ub + \yb \cdot \vb ) }    \xb^\kb  \yb^\lb      d \xb d \yb \bigg|    \leq  \iint_{ || \zb || \geq \delta_4 } e^{ -  \pi^2 \xi_{min} \theta \log \log T || \zb ||^2  }   || \zb ||^{\Kcal( \kb+\lb)}      d \xb d \yb   \\
& \ll  \int_{\delta_4}^\infty e^{ -  (\pi^2 \xi_{min} \theta \log \log T ) r^2  }   r^{\Kcal( \kb+\lb) + 2J-1 } dr    \\
& \ll   \frac{1}{ (\pi^2 \xi_{min} \theta \log \log T )^{ \frac{ \Kcal ( \kb + \lb )}{2} + J  } }  \int_{ \pi \delta_4 \sqrt{ \xi_{min} \theta \log \log T }}^\infty e^{ -    r^2  }   r^{\Kcal( \kb+\lb) + 2J-1 } dr  
\end{align*}
by the change of variables to the polar coordinates.
By the Cauchy-Schwarz inequality we have
$$ \int_X^\infty e^{-r^2 } r^M dr \leq \sqrt{ \int_X^\infty e^{-r^2 }rdr \int_0^\infty e^{-r^2} r^{2M-1} dr } = \frac{ \sqrt{ (M-1)!}}{2} e^{- \frac12 X^2 } . $$
Hence, it follows from  Lemma \ref{lem bkl} and the above estimations   that 
\begin{align*}
  H_T ( \ub, \vb) = &    \sum_{\Kcal(\kb+ \lb) \leq \epsilon \log \log T }  ( 2 \pi i )^{\Kcal(\kb + \lb)}  b_{\kb, \lb }  \int_{\mathbb{R}^J }\int_{\mathbb{R}^J } e^{Q_T (\zb)  - 2 \pi i ( \xb \cdot \ub + \yb \cdot \vb ) }    \xb^\kb  \yb^\lb      d \xb d \yb  \\
  &+ O \bigg( \frac{1}{  ( \log T)^{  \frac12 \pi^2 \delta_4^2 \xi_{min} \theta }}  \sum_{\Kcal(\kb+ \lb) \leq \epsilon \log \log T }  \bigg(   \frac{ 2 \pi  }{\delta_3 } \bigg)^{\Kcal(\kb + \lb)} \frac{\sqrt{ (\Kcal(\kb + \lb ) + 2J-2)!   } }{(\pi^2 \xi_{min} \theta \log \log T )^{ \frac{ \Kcal ( \kb + \lb )}{2} + J  } }  \bigg)  \\
  & + O \bigg(  \frac{1}{ ( \log T)^\kappa} \bigg) .
  \end{align*}
By Stirling's formula the $\kb, \lb$-sum in the above $O$-term is 
\begin{align*}
& \ll  \sum_{\Kcal(\kb+ \lb) \leq \epsilon \log \log T }  \bigg(   \frac{ 2 \pi  }{\delta_3 } \bigg)^{\Kcal(\kb + \lb)} \frac{1 }{(\pi^2 \xi_{min} \theta \log \log T )^{ \frac{ \Kcal ( \kb + \lb )}{2} + J  } }  \bigg(  \frac{ 2 \epsilon \log \log T}{e} \bigg)^{ \frac{ \Kcal(\kb + \lb)}{2} + J - \frac34}\\ 
& \ll  \sum_{ \kb , \lb  }  \bigg(   \frac{ 2 \sqrt{2 \epsilon}  }{\delta_3 \sqrt{ \xi_{min} \theta e  }    } \bigg)^{\Kcal(\kb + \lb)}   \leq  \sum_{\kb, \lb  }  \bigg( \frac12  \bigg)^{\Kcal(\kb + \lb)} = 2^{2J},  
\end{align*}
provided that $    0< \epsilon \leq \frac{1}{ 32} \delta_3^2 \xi_{min} \theta e $.
 With this choice of $\epsilon$, we have
\begin{align*}
  H_T ( \ub, \vb) = &    \sum_{\Kcal(\kb+ \lb) \leq \epsilon \log \log T }  ( 2 \pi i )^{\Kcal(\kb + \lb)}  b_{\kb, \lb }  \int_{\mathbb{R}^J }\int_{\mathbb{R}^J } e^{Q_T (\zb)  - 2 \pi i ( \xb \cdot \ub + \yb \cdot \vb ) }    \xb^\kb  \yb^\lb      d \xb d \yb  \\ 
  & + O \bigg(  \frac{1}{ ( \log T)^\kappa} \bigg)  
  \end{align*}  
for some $\kappa>0$

It remains to calculate the above integral. We first write it as  repeated integrals
\begin{align*}
 &  \int_{\mathbb{R}^J }\int_{\mathbb{R}^J } e^{Q_T (\zb)  - 2 \pi i ( \xb \cdot \ub + \yb \cdot \vb ) }  \xb^\kb \yb^\lb  d \xb d \yb \\
 & =   \prod_{j=1}^J  \int_{\mathbb{R}} \int_{\mathbb{R}} e^{  -  \psi_{j,T} \pi^2 ( x_j^2 + y_j^2 )   - 2 \pi i ( x_j u_j  + y_j v_j  ) }  x_j^{k_j} y_j^{\ell_j}  dx_j  d y_j \\
 & =   \prod_{j=1}^J   \int_{\mathbb{R}} e^{  -  \psi_{j,T}  \pi^2 x_j^2    - 2 \pi i   x_j u_j    }  x_j^{k_j}   dx_j    \int_{\mathbb{R}}   e^{  -  \psi_{j,T} \pi^2  y_j^2    - 2 \pi i  y_j v_j }   y_j^{\ell_j}    d y_j.
 \end{align*}
Each integral can be written in terms of the Hermite polynomials defined in \eqref{def Hcal}. Since
\begin{align*}
 \int_\mathbb{R}  e^{ -  \psi \pi^2  x^2   - 2 \pi i   x u    } x^{k}   dx  &= \frac{1}{(- 2 \pi i )^{k}}\frac{d^k }{d u^k } \int_\mathbb{R}  e^{ - \psi \pi^2   x^2   - 2 \pi i   x u    }    dx  \\
 & = \frac{1}{(- 2 \pi i )^{k}}\frac{d^k }{d u^k } \frac{1}{ \sqrt{  \pi  \psi  }} e^{-  \frac{   u^2  }{\psi }} \\
 & = \frac{1}{( 2 \pi i )^{k} \sqrt{ \pi}  \sqrt{ \psi }^{k+1}    }    e^{-  \frac{  u^2  }{\psi  }}   \Hcal_k \bigg( \frac{  u }{ \sqrt{ \psi  }} \bigg)  , 
\end{align*}  
  we have
\begin{align*}
& \int_{\mathbb{R}^{J} }\int_{\mathbb{R}^{J} } e^{Q_T (\zb)  - 2 \pi i ( \xb \cdot \ub + \yb \cdot \vb ) }    \xb^\kb   \yb^\lb    d \xb d \yb \\
&  =   \prod_{j =1}^J   \frac{1}{ \pi ( 2 \pi i )^{k_j+\ell_j }   \sqrt{ \psi_{j, T}}^{k_j + \ell_j +2}    }    e^{-  \frac{  u_j ^2  + v_j^2  }{\psi_{j,T} }}   \Hcal_{k_j}  \bigg( \frac{  u_j }{ \sqrt{ \psi_{j,T } }} \bigg)  \Hcal_{\ell_j}  \bigg( \frac{  v_j }{ \sqrt{ \psi_{j,T } }} \bigg)  . 
\end{align*}
Thus, we have
\begin{align*}
  H_T ( \ub, \vb) = &    \sum_{\Kcal(\kb+ \lb) \leq \epsilon \log \log T }    b_{\kb, \lb }     \prod_{j =1}^J   \frac{1}{ \pi    \sqrt{ \psi_{j, T}}^{k_j + \ell_j +2}    }    e^{-  \frac{  u_j ^2  + v_j^2  }{\psi_{j,T} }}   \Hcal_{k_j}  \bigg( \frac{  u_j }{ \sqrt{ \psi_{j,T } }} \bigg)  \Hcal_{\ell_j}  \bigg( \frac{  v_j }{ \sqrt{ \psi_{j,T } }} \bigg) \\ 
  & + O \bigg(  \frac{1}{ ( \log T)^\kappa} \bigg)  
  \end{align*}  
for some $\epsilon, \kappa>0$. This completes the proof of Proposition \ref{prop}.

\section{Proofs of lemmas}\label{sec 3}
We prove  Lemma \ref{lemma Rpst} in Section \ref{sec proof lemma Rpst}, Lemma \ref{lemma asymp InsT} in Section \ref{sec proof lemma asymp InsT} and Lemma \ref{lem bkl} in Section \ref{sec proof lem bkl}. In the proofs, we need the inequalities
\begin{equation}\label{beta bound 0}
  |\beta_{L_j}(p^k)| \leq \frac{d}{k}   p^{k\eta}  \quad \mathrm{for~} k \geq 1,
  \end{equation}
\be\label{beta bound 1} 
   | \beta_{L_j } (p^k) | \leq    \frac1k  \sum_{i=1}^d  | \alpha_{j,i}(p)|^k \leq \frac{p^{(k-2)\eta}}{k} \sum_{i=1}^d  |\alpha_{j,i} (p)|^2  \quad \mathrm{for~} k \geq 2 
 \ee 
 and  
\be\label{beta bound 2}
| \beta_{L_j } (p ) |^2  \leq    \bigg(  \sum_{i=1}^d  | \alpha_{j,i}(p)| \bigg)^2    \leq d  \sum_{i=1}^d  |\alpha_{j,i} (p)|^2     ,
 \ee
 which follows by \eqref{eqn Ljpk} and assumpion A1.

\subsection{Proof of Lemma \ref{lemma Rpst}}\label{sec proof lemma Rpst}

 By \eqref{def gjps} and \eqref{beta bound 0}  there is a constant $ C_1 := C_{1, d, \eta} >0$ such that 
\begin{equation}\label{gjp bound 1}
 | g_{j,p} (  \sigma_T  ) |  \leq \sum_{k=1}^\infty \frac{d}{k}\frac{ p^{k\eta}}{ p^{\frac{k}{2}}}  \leq \frac{C_1 }{ p^{\frac12 - \eta}  }
\end{equation}
for every prime $p$ and $j = 1, \ldots, J$. 
By \eqref{def apskl}, \eqref{def Rpstzb} and \eqref{gjp bound 1} we obtain
$$  |R_{p, \sigma_T }(\zb) |  \leq   \sum_{\kb \neq 0}  \sum_{\lb \neq 0 }   \frac{1}{\kb! \lb!}    \bigg( \pi || \zb ||  \frac{C_1}{ p^{\frac12 - \eta}  } \bigg)^{\Kcal( \kb + \lb )}  =   \bigg( \exp \bigg( J \frac{  \pi C_1 ||\zb ||}{ p^{\frac12 - \eta}} \bigg) -1 \bigg)^2.$$
Thus, there exists a constant $C_2 := C_{2, d, J, \eta}>0 $ such that
$$ |R_{p, \sigma_T } ( \zb)| \leq  \frac{ C_2 }{ p^{1-2\eta}} || \zb ||^2 \leq \frac{ C_2 }{ 2^{1-2\eta}} || \zb ||^2 $$
for $ || \zb || \leq 1 $ and every prime $p$. Therefore, there exists a constant $\delta_1 >0$ such that
$$ | R_{ p, \sigma_T } ( \zb) | \leq \frac12 $$
for  $ ||\zb ||\leq \delta_1 $ and every prime $p$.

\subsection{Proof of Lemma \ref{lemma asymp InsT}}\label{sec proof lemma asymp InsT}

We first find an useful expression  
\begin{multline}\label{eqn instzb 1}
I_{n, \sigma  } ( \zb ) = ( \pi i )^n  \sum_{1 \leq m \leq    n/2    } \frac{ (-1)^{m-1}}{m} \sum_{   \substack{  \kb_1 , \ldots , \kb_m , \lb_1 , \ldots, \lb_m  \neq 0 \\ \Kcal( \kb_1 + \cdots + \kb_m + \lb_1 + \cdots + \lb_m ) = n }}  \frac{    \overline{\zb}^{\kb_1 + \cdots + \kb_m} \zb^{ \lb_1 + \cdots + \lb_m }}{ \kb_1 ! \cdots \kb_m ! \lb_1 ! \cdots \lb_m ! } \\
\times \sum_p A_{p, \sigma  } ( \kb_1 , \lb_1 ) \cdots A_{p, \sigma  } ( \kb_m , \lb_m )
 \end{multline}
 by \eqref{def Rpstzb}, \eqref{def Bskl} and \eqref{def Inszb}.
Here, the sum over $m$ is  $ 1 \leq m \leq    n/2    $ because 
$$ n= \Kcal( \kb_1 + \cdots + \kb_m + \lb_1 + \cdots + \lb_m )  \geq 2m $$
for   $\kb_1 , \ldots , \kb_m , \lb_1 , \ldots, \lb_m  \neq 0$.

The asymptotic \eqref{eqn I2sTzb asymp} of $I_{2, \sigma_T }( \zb)$ is known before. See (7.16) of \cite[Lemma 7.3]{LL}.  We next prove 
\begin{equation}\label{eqn Cj1j2 conj}
 \overline{  C_{j_1, j_2} } = C_{j_2, j_1}.
 \end{equation}
  We have  
  \begin{equation}\label{Apskl conj}
   \overline{ A_{p, \sigma } ( \kb, \lb) } = A_{p, \sigma } ( \lb, \kb) 
\end{equation}
 by \eqref{def apskl}. By   \eqref{eqn instzb 1} we also have
  \begin{equation}\label{I2sz conj}
   \overline{ I_{2, \sigma } ( \zb )  } =    I_{2, \sigma } ( \zb )   . 
   \end{equation}
So we obtain \eqref{eqn Cj1j2 conj} by  \eqref{eqn I2sTzb asymp} and \eqref{I2sz conj}.

For the case $n>2 $,  we observe that $ A_{p, \sigma} ( \kb, \lb)$ for a real $\sigma $ can be extended to an analytic function in a complex variable $s$ via 
  \begin{equation}\label{def Apskl}
 A_{ p, s  } (\kb, \lb ) = \E \bigg[  \prod_{j=1}^J \bigg(   \sum_{k=1 }^\infty  \frac{  \beta_{L_j} ( p^k ) \X(p)^k }{ p^{k s  }}    \bigg)^{k_j }  \bigg(   \sum_{k=1 }^\infty  \frac{ \overline{ \beta_{L_j} ( p^k )} \overline{ \X(p)^k} }{ p^{k s  }}   \bigg)^{\ell_j } \bigg] . 
 \end{equation}
  This observation essentially leads us to prove  the following lemma.
\begin{lem}\label{lem fs}
Let $\eta $ be the constant in assumption A1 and assume $ \Kcal ( \kb_1 + \cdots + \kb_m + \lb_1 + \cdots + \lb_m  )  = n   \geq 3$. The Dirichlet series
$$
 f(s) := \sum_p A_{p, s } ( \kb_1 , \lb_1 ) \cdots A_{p, s} ( \kb_m , \lb_m ) 
$$
 is absolutely convergent for $\re(s) \geq  \frac{ 5+2 \eta}{12} $. Moreover,  there exists a constant $C_3= C_{ 3, J, d, \eta} >0  $ such that 
  $$ |f(s) | \leq C_3^n  $$
 for $ \re(s) \geq \frac{5 + 2 \eta}{12} $ and 
$$   |   f( \sigma_T ) - f( \tfrac12 )   | \leq  \frac{ C_3^n }{ ( \log T)^\theta}  . $$
  \end{lem}

\begin{proof}
 We first show that there is a constant $ C_4  >0$ such that 
$$  |f(s) | \leq C_4^n $$
  for $ \re ( s )  \geq  \frac{ 5+2 \eta}{12}$.
By \eqref{def Apskl} we find that
$$  | A_{p, s} ( \kb, \lb ) |   \leq \bigg(    \sum_{k=1}^\infty   \frac{   \max_{j\leq J } | \beta_{ L_j } ( p^k ) | }{ p^{k\re(s) }} \bigg)^{\Kcal( \kb+\lb) } . $$
 Thus, we have
\begin{align}
| f(s)|   & \leq  \sum_p \bigg(    \sum_{k=1}^\infty   \frac{  \max_{j\leq J } | \beta_{ L_j } ( p^k ) | }{ p^{k\re(s)}} \bigg)^n  \notag  \\
& \leq 2^n \sum_p  \bigg(    \frac{   \max_{j\leq J } | \beta_{ L_j } ( p  ) |   }{ p^{ \re(s)}} \bigg)^n + 2^n \sum_p \bigg( \sum_{k=2}^\infty   \frac{  \max_{j\leq J } | \beta_{ L_j } ( p^k ) | }{ p^{k\re(s)}} \bigg)^n  \label{fs sum}.
\end{align}

The first sum on the right hand side of \eqref{fs sum} is
\begin{align*}
 \sum_p      \frac{  \big( \max_{j\leq J }  | \beta_{ L_j } ( p  ) | \big)^n  }{ p^{ n \re(s) }}&  \leq \sum_p      \frac{  (dp^\eta)^{n-2}  \big( \max_{ j \leq J} d  \sum_{i=1}^d | \alpha_{j,i } ( p  ) |^2 \big)   }{ p^{ n \re(s) }}  \\
   & \leq  d^{n-1} \sum_p      \frac{    \sum_{j=1}^J \sum_{i=1}^d  | \alpha_{j,i } ( p  ) |^2     }{ p^{1+\varepsilon}} \leq C_5^n  
\end{align*}
for $ \re(s) \geq \frac{5 + 2 \eta}{12} $  by \eqref{beta bound 0} and \eqref{beta bound 2}, where $ \varepsilon =  \frac14 - \frac{\eta}{2} >0 $ and 
$$ C_5  := \max\bigg\{  d ,   \sum_p      \frac{   \sum_{j=1}^J \sum_{i=1}^d  | \alpha_{j,i } ( p  ) |^2      }{ p^{1+\varepsilon}} \bigg\} . $$
Note that the last $p$-sum is convergent by assumption A4 and a partial summation.
 The second sum on the right hand side of \eqref{fs sum} is
\begin{align*}
\sum_p \bigg( \sum_{k=2}^\infty   \frac{  \max_{ j \leq J} | \beta_{ L_j } ( p^k ) | }{ p^{k\re(s) }} \bigg)^n &  \leq  \sum_p \bigg( \sum_{k=2}^\infty   \frac{  \max_{ j \leq J}  \sum_{i=1}^d | \alpha_{j,i}(p)  |^2  }{ k  p^{k\re(s) -(k-2)\eta  }} \bigg)^n \\
 &  \leq  \sum_p \bigg(    \frac{  \max_{ j \leq J}  \sum_{i=1}^d | \alpha_{j,i}(p)  |^2  }{   p^{2\re(s)  }}   \frac{1}{2 }   \frac{1}{ 1- \frac{1}{ p^{\re(s) - \eta}}} \bigg)^n \\
 &  \leq \bigg( \frac{1}{2 }   \frac{1}{ 1- \frac{1}{ p^{\frac{5}{12}(1-2\eta) }}} \bigg)^n \sum_p     \frac{  (d p^{2 \eta} )^{n-1} \max_{ j \leq J}  \sum_{i=1}^d | \alpha_{j,i}(p)  |^2  }{   p^{2n \re(s)  }}    \\
 &  \leq \bigg( \frac{1}{2 }   \frac{1}{ 1- \frac{1}{ p^{\frac{5}{12}(1-2\eta) }}} \bigg)^n d^{n-1}  \sum_p     \frac{    \sum_{j=1}^J   \sum_{i=1}^d | \alpha_{j,i}(p)  |^2  }{   p^{1+6 \varepsilon  }} \leq C_6^n    
\end{align*}
for $ \re(s) \geq \frac{5 + 2 \eta}{12} $ by   \eqref{beta bound 1}, where
$$ C_6 := \frac{1}{2 }   \frac{1}{ 1- \frac{1}{ p^{\frac{5}{12}(1-2\eta) }}}   \max\bigg\{  d ,  \sum_p    \frac{    \sum_{j=1}^J \sum_{i=1}^d | \alpha_{j,i}(p)  |^2     }{ p^{1+6 \varepsilon }}  \bigg\} . $$
We choose $C_4 = 2 (C_5 + C_6) $, then we have
\begin{equation}\label{fsC4 bound}
  |f(s) | \leq C_4^n  
\end{equation}
 for $ \re(s) \geq \frac{5 + 2 \eta}{12} $. One can easily see in the above estimations that $f(s)$ is absolutely convergent for $ \re(s) \geq \frac{5 + 2 \eta}{12} $.

  Let $ \varepsilon_1 = \frac12  - \frac{5 + 2 \eta}{12} >0 $. Since
 \begin{align*}
 f(\sigma_T ) - f( \tfrac12 ) = \int_{1/2}^{\sigma_T}  f'(u) du =  \int_{1/2}^{\sigma_T} \frac{1}{ 2 \pi i } \int_{ | z - u | = \varepsilon_1}  \frac{ f(z)}{ ( z-u)^2 } dz du ,
 \end{align*}
we obtain
 \begin{equation}\label{eqn fsf12}
| f(\sigma_T ) - f( \tfrac12 ) | \leq  ( \sigma_T - \frac12)   \frac{1}{ \varepsilon_1  } \sup_{ \re(z) \geq \frac12 - \varepsilon_1 }  | f(z) |  \leq     \frac{C_4^n }{ \varepsilon_1 ( \log T)^\theta} 
 \end{equation}
 by \eqref{fsC4 bound}. Let $ C_3  =   C_4 / \varepsilon_1 > C_4 $, then \eqref{fsC4 bound} and \eqref{eqn fsf12} imply  both inequalities in the lemma.
\end{proof}

Therefore by Lemma \ref{lem fs}, \eqref{eqn instzb 1} and Stirling's formula we have
\begin{align*}
  |I_{n,\sigma} ( \zb )|  &  \leq  ||\zb||^n    ( \pi   C_3)^n \sum_{m \leq n/2} \frac{1}{m} \sum_{ \Kcal(\kb_1 + \cdots + \kb_m + \lb_1 + \cdots + \lb_m ) = n }    \frac{  1 }{ \kb_1 ! \cdots \kb_m ! \lb_1 ! \cdots \lb_m !  }   \\
  &= ||\zb||^n    ( \pi   C_3)^n   \sum_{m \leq n/2} \frac{1}{m}      \frac{  (2mJ)^n }{n!}  \\
    & \leq ||\zb||^n    ( J \pi   C_3)^n   \frac{ n^n}{n!} \\
  & \leq ||\zb||^n    ( J \pi   C_3 e)^n 
  \end{align*}
for $\sigma \geq \frac{ 5+2 \eta}{12}  $  and $n > 2 $.
Similarly, we have
$$ |  I_{n , \sigma_T } ( \zb ) -  I_{n , 1/2 } ( \zb) | \leq   \frac{  ||\zb||^n    ( J \pi   C_3 e)^n  }{   ( \log T)^\theta}   $$
for  $ n >  2 $. Therefore, Lemma \ref{lemma asymp InsT} holds with a constant  
\begin{equation}\label{def C}
 C= J \pi C_3 e . 
\end{equation}

\subsection{Proof of Lemma \ref{lem bkl}}\label{sec proof lem bkl}

We first consider $G( \xb, \yb)$ in \eqref{def bkl} as a function in complex variables $x_1, \ldots, x_J , y_1, \ldots , y_J$. We replace $ x_j$ by $ \frac{x_j}{ 2 \pi i }$ and $ y_j $ by $\frac{y_j}{ 2 \pi i }$ for $ j = 1 , \ldots , J $ in \eqref{def bkl}, then we obtain that 
\begin{equation}\label{def bkl 2}
    \sum_{\kb, \lb }   b_{\kb, \lb} \xb^\kb \yb^\lb  =   \sum_{r=0}^\infty \frac{1}{ r!}  \bigg(    \sum_{ n=2}^\infty        I_{n  } ( \zb)  ( 2 \pi i )^{-n}  \bigg)^r  .
    \end{equation}
Now we consider $x_1, \ldots, x_J , y_1, \ldots , y_J$ as real variables.  By \eqref{eqn instzb 1} and \eqref{Apskl conj} we have
$$   \overline{ I_{n, \sigma } ( \zb )  (2 \pi i)^{-n} } =    I_{n, \sigma } ( \zb ) ( 2 \pi i )^{-n}  , $$
which  implies that  $I_{n, \sigma} ( \zb ) ( 2 \pi i )^{-n} $ is a polynomial in real variables $ x_1, \ldots, x_J$, $ y_1, \ldots ,y_J$ with real coefficients.  Since $ I_n ( \zb ) ( 2 \pi i )^{-n}$ is also a homogeneous polynomial in $x_1 , \ldots, x_J$, $y_1 , \ldots , y_J$ of degree $n$ with real coefficients, we obtain by comparing coefficients in \eqref{def bkl 2}  that $b_{\kb, \lb} \in \mathbb{R}$, $ b_{0,0}=1 $ and $ b_{\kb, \lb} = 0$ for $\Kcal(\kb  + \lb ) = 1 $. 

It remains to prove the inequality \eqref{bkl bound}. Again we consider $G(\xb, \yb) $ defined in \eqref{def bkl} as an analytic function in complex variables $ x_1 , \ldots , x_J, y_1, \ldots, y_J$. Assume that 
$$ \sup \{ |x_1| , \ldots, |x_J| , |y_1 | , \ldots , |y_J | \} \leq \frac{ \delta_2 }{ 2 \sqrt{J}} . $$
Then we see that 
$$ |  I_2 ( \zb) | \leq \sum_{j_1, j_2 =1}^J  | C_{j_1, j_2}|   \frac{\delta_2^2}{4J} \leq \frac1{16} $$
by \eqref{def I2zb} and \eqref{def delta2}. For $ n \geq 3 $ we have
\begin{align*}
  |I_{n} ( \zb )|  &  \leq   \bigg(  \frac{ \delta_2\pi   C_3}{ \sqrt{J}} \bigg)^n   \sum_{m \leq n/2} \frac{1}{m} \sum_{ \Kcal(\kb_1 + \cdots + \kb_m + \lb_1 + \cdots + \lb_m ) = n }    \frac{  1 }{ \kb_1 ! \cdots \kb_m ! \lb_1 ! \cdots \lb_m !  }   \\
  & \leq    (  \delta_2 \sqrt{J} \pi   C_3 e)^n \leq ( \delta_2 C)^n  \leq 2^{-n}
  \end{align*}
 by \eqref{def Inzb}, \eqref{def delta2}, \eqref{eqn instzb 1}, \eqref{def C}  and Lemma \ref{lem fs}. Thus,
 $$|  G(\xb, \yb) |   \leq \sum_{r=0}^\infty \frac{1}{r!} \bigg( \sum_{n=2}^\infty |I_n ( \zb)| \bigg)^r \leq  \sum_{r=0}^\infty \frac{1}{r!} 2^{-r} = \sqrt{e} . $$
 Let $ 0 < \frac{ \delta_3 }{ 2 \pi } = \delta'_3 <  \frac{ \delta_2 }{ 2 \sqrt{J}} $. Since  
$$   b_{\kb, \lb} = \frac{1}{ ( 2 \pi i )^{\Kcal( \kb+ \lb)+ 2J } }    \int_{ | x_1|= \delta'_3} \cdots \int_{ |x_J| =\delta'_3 } \int_{ |y_1 | = \delta'_3} \cdots \int_{ | y_J | = \delta'_3 }  \frac{ G( \xb, \yb)}{  \xb^\kb \yb^\lb}    \frac{ dy_J}{y_J} \cdots \frac{ dy_1}{y_1} \frac{ dx_J}{x_J} \cdots \frac{ dx_1}{ x_1} $$
by Cauchy's integral formula, we obtain
$$ | b_{\kb , \lb} | \leq    \frac{\sqrt{e} }{ (2 \pi  \delta'_3)^{\Kcal(\kb+\lb)}} =  \frac{\sqrt{e} }{ \delta_3^{\Kcal(\kb+\lb)}}.$$

\section*{Acknowledgements}

This work has been supported by the National Research Foundation of Korea (NRF) grant funded by the Korea government (MSIP) (No. 2019R1F1A1050795).

\end{document}